\documentclass{amsart}
\usepackage{amsaddr}
\usepackage[english]{babel}
\usepackage{comment}
\usepackage[letterpaper,top=2cm,bottom=2cm,left=3cm,right=3cm,marginparwidth=1.75cm]{geometry}
\usepackage{todonotes}

\usepackage{amsmath}
\usepackage{amssymb}
\usepackage{graphicx}
\usepackage[colorlinks=true, allcolors=blue]{hyperref}
\usepackage{cleveref}
\usepackage{algorithm}
\usepackage{algpseudocode}
\usepackage{caption}
\usepackage{subcaption}
\usepackage{xcolor}

\newcommand{\N}{N}

\DeclareMathOperator{\indeg}{indeg}
\DeclareMathOperator{\outdeg}{outdeg}

\usepackage{tikz}
\usepackage{tkz-graph}
\usetikzlibrary{shapes,positioning,calc}
\tikzstyle{tre}=[circle,draw,inner sep = 0pt, minimum size=1.5mm]
\tikzstyle{pe}=[circle,draw,inner sep=0pt,minimum size=1.5mm]
\tikzstyle{mi}=[circle,draw,inner sep=0pt,minimum size=1.5mm, fill = black!60]
\tikzstyle{mi2}=[circle,color = lightgray,draw,inner sep=0pt,minimum size=1.5mm, fill = lightgray]
\tikzstyle{mini}=[circle,draw,inner sep=0pt,minimum size=0.8mm, fill = black!60]
\tikzstyle{minil}=[circle,draw,color = lightgray,inner sep=0pt,minimum size=0.8mm, fill = lightgray]
\tikzstyle{minir}=[circle,color = red,draw,inner sep=0pt,minimum size=0.8mm, fill = red!60]
\tikzstyle{minib}=[circle,color = blue,draw,inner sep=0pt,minimum size=0.8mm, fill = blue!60]
\tikzstyle{mimarc}=[circle,draw,inner sep=0pt,minimum size=1.5mm, fill = red!60]
\tikzstyle{bxmarc}=[rectangle,draw,inner sep=0pt,line width=0.2pt,minimum size=1.5mm, fill = red!60]
\tikzstyle{bb}=[circle,draw,inner sep=0pt,line width=1.5pt,minimum size=4.5mm]
\tikzstyle{rbx}=[circle,fill,draw,inner sep=0pt,minimum size=1.5mm]
\tikzstyle{bx}=[rectangle,draw,inner sep=0pt,line width=0.2pt,minimum size=1.5mm, fill = black!60]
\tikzstyle{bxb}=[rectangle,draw,inner sep=0pt,line width=0.2pt,minimum size=1.5mm, fill = blue!70]
\tikzstyle{minibx}=[rectangle,draw,inner sep=0pt,line width=0.2pt,minimum size=0.8mm, fill = black!60]
\tikzstyle{minilbx}=[rectangle,color = lightgray,draw,inner sep=0pt,line width=0.2pt,minimum size=0.8mm, fill = lightgray]
\tikzstyle{minirbx}=[rectangle,color = red,draw,inner sep=0pt,line width=0.2pt,minimum size=0.8mm, fill = red!60]
\tikzstyle{minibbx}=[rectangle,color = blue,draw,inner sep=0pt,line width=0.2pt,minimum size=0.8mm, fill = blue!60]
\tikzstyle{label}=[circle,inner sep=0pt,minimum size=0.2mm, fill = white]
\tikzstyle{ppalpath}=[style={->},
decoration={snake,pre length=0.05cm, post length=0.05cm,
amplitude=.4mm,segment length=2mm},
decorate]
\tikzset{
    between/.style args={#1 and #2}{
         at = ($(#1)!0.5!(#2)$)
    }
}
\tikzset{
    amunt/.style args={#1 and #2}{
         at = ($(#1)!0.25!(#2)$)
    }
}

\newtheorem{thm}{Theorem}
\newtheorem{prop}[thm]{Proposition}
\newtheorem{lem}[thm]{Lemma}
\newtheorem{cor}[thm]{Corollary}

\newtheorem{defn}[thm]{Definition}
\newtheorem{exm}[thm]{Example}
\newtheorem{rmk}[thm]{Remark}

\title[Characterization of tree-child networks in terms of $\mu$ vectors]{Characterization of tree-child networks \\ in terms of $\mu$ vectors}

\author{%
    Toni Fuentes\textsuperscript{a},   Vincent Moulton \textsuperscript{b}, Katharina T.~Huber \textsuperscript{b}, Joan Carles Pons\textsuperscript{a,*}.
     \\[1em]
    \textsuperscript{a} Department of Mathematics and Computer Science, \\ 
    Universitat de les Illes Balears, Spain \\[1em]
        \textsuperscript{b} School of Computing Sciences, 
University of East Anglia, Norwich, NR4~7TJ, UK \\[1em]
    \textsuperscript{*} Corresponding author: joancarles.pons@uib.es
}

\begin{document}

\begin{abstract}
We characterize tree-child phylogenetic networks in terms
of their \(\mu\)-representations. First, we give a structural characterization of
tree-child networks by means of ordered tree-path decompositions. We then translate
this decomposition into a set of purely vectorial conditions on finite subsets
\(M\subseteq\mathbb N^n\). We prove that such a set \(M\) is the
\(\mu\)-representation of a tree-child phylogenetic network if and only if it is
tree-child \(\mu\)-compatible.  This provides a feasibility criterion for tree-child
\(\mu\)-representations which can be used as a basis for reconstruction and further algorithmic
applications. Note that this paper presents results arising from ongoing research
on tree-child networks and that the results will be further developed and placed 
into proper context in subsequent versions.
\end{abstract}

\maketitle

\section{Introduction}


Tree-child phylogenetic networks form a well-studied class of rooted
phylogenetic networks. 
One of the standard encodings of tree-child networks is the
\(\mu\)-representation introduced by Cardona et al.~\cite{cardona2008comparison}.
For each vertex \(v\), the vector \(\mu(v)\) records the number of directed paths
from \(v\) to each labelled leaf. It is known that the \(\mu\)-representation
determines tree-child networks up to isomorphism. However, this does not solve the
inverse problem of deciding which finite sets of vectors actually arise as
\(\mu\)-representations of tree-child networks.

The aim of this paper is to solve this feasibility or encoding problem for tree-child
networks. We first characterize tree-child networks by
ordered tree-path decompositions. Such a decomposition expresses the network as a
collection of directed paths ending at the leaves, together with additional
 arcs from vertices of one path to the maximal vertex of a
later path.

We then translate this structure into vectorial conditions on a
finite set \(M\subseteq\mathbb N^n\). This leads to the notion of
tree-child \(\mu\)-compatibility. Our main result states that \(M\) is the
 \(\mu\)-representation of a tree-child phylogenetic network if and only if \(M\) is tree-child \(\mu\)-compatible (see Theorem~\ref{thm:tc-mu-characterization}). The characterization
is constructive: the conditions define a reconstruction digraph, and we
prove that this digraph realizes \(M\).

The rest of this paper is organized as follows. Section~2 recalls the basic definitions.
Section~3 proves the characterization of tree-child networks by ordered
tree-path decompositions. Section~4 gives the vectorial characterization in terms
of \(\mu\)-compatible sets and proves the reconstruction theorem.
The last section describes some future directions that we
plan to follow up in subsequent versions of this paper.

\section{Preliminaries}

\subsection{Basic definitions}

For \(n\geq 1\), let \([n]:=\{1,2,\ldots,n\}\). Throughout the paper, we use \(\mathbb N=\{0,1,2,\ldots\}\). A \emph{phylogenetic network} (or \emph{network}) \(N=(V,E)\) on \([n]\) is a finite rooted
directed acyclic graph with a unique root \(\rho\), satisfying \(\indeg(\rho)=0\),
such that every vertex is reachable from \(\rho\), and whose leaves are bijectively
labelled by the elements of \([n]\). We denote by \(l_i\) the leaf labelled by \(i\).
Throughout the paper, isomorphisms of phylogenetic networks are understood to preserve leaf labels.

A vertex \(v\neq \rho\) is a \emph{tree vertex} if \(\indeg(v)=1\), and a
\emph{reticulation vertex} if \(\indeg(v)>1\). A vertex is a \emph{leaf} if
\(\outdeg(v)=0\), and an \emph{inner vertex} otherwise.

We will denote by \(L(\N)\) the set of leaf nodes and by $R(\N)$ the set of reticulations of \(\N\). 
We do not require phylogenetic networks to be \emph{binary}. In the binary case,
reticulation vertices have in-degree two and out-degree one, and inner tree vertices
have out-degree two. Here we allow tree vertices to have arbitrary out-degree, and
reticulation vertices to have arbitrary in-degree greater than one and arbitrary
out-degree. In particular, reticulation vertices may also be leaves.

Depending on the topological conditions imposed on the underlying directed graph,
one obtains different classes of phylogenetic networks \cite{kong2022classes}. In this
paper we focus on \emph{tree-child networks}, namely phylogenetic networks in
which every inner vertex has at least one tree vertex as a child \cite{cardona2008comparison}.

\subsection{Paths}

A \emph{path}, or \emph{directed path}, in a network \(\N=(V,E)\) is a non-empty
sequence of vertices
$(v_1,v_2,\ldots,v_k)$
such that \((v_i,v_{i+1})\in E\) for all \(1\le i\le k-1\). A path is called
\emph{trivial} if it consists of a single vertex.
A non-trivial path \((v_1,v_2,\ldots,v_k)\) is called \emph{elementary} if
$\outdeg(v_i)=\indeg(v_{i+1})=1$, for all $1\le i\le k-1$. We write
$u\rightsquigarrow v$
to denote either a path from \(u\) to \(v\), or simply the existence of such a
path. The path may be trivial, so \(u\rightsquigarrow u\) always holds.

 Since $N$ is acyclic, we can  define the following partial order on $V$. For
\(u,v\in V\), we write
$
u\ge v
$
if there exists a (possibly trivial) path from \(u\) to \(v\). We write
\(u>v\) if \(u\ge v\) and \(u\ne v\). Thus, \(u>v\) means that \(u\) is a strict
ancestor of \(v\).

It is common in the study of phylogenetic networks to exclude elementary paths
\cite{kong2022classes}. Unless otherwise stated, we also impose this restriction
throughout this paper. Since every non-trivial elementary path contains an arc
\((u,v)\) satisfying
$\outdeg(u)=1$ 
and 
$\indeg(v)=1$,
this is equivalent to requiring that no such arc exists in \(\N\).
Note that by this definition, an edge connecting a reticulation node with out-degree 1 and a tree node with in-degree 1 is also considered an elementary path, even though it is common not to consider this particular case as an elementary path. We consider this an elementary path since both the reticulation and its child will have associated the same \(\mu\)-vector. In fact, in the particular case of tree-child networks, any node with out-degree 1 will be included in an elementary path, since its descendant will always be a tree node with in-degree 1. In order to address this issue, we are compressing all these edges into a single vertex, allowing reticulation nodes to have any in-degree greater than 1 and any out-degree other than 1 (including 0 in the case they are also leaf nodes).

A path $P=(u_1,u_2,\ldots,u_k)$
is a \emph{tree-path} if every edge of \(P\) ends in a tree vertex; equivalently,
\(u_2,\ldots,u_k\) are tree vertices. Notice that the first vertex \(u_1\) may be the
root, a tree vertex, or a reticulation vertex. This differs slightly from other
definitions in the literature \cite{cardona2008comparison}, where all vertices in a
tree-path are required to be tree vertices.

The absence of elementary paths has a simple consequence for tree-child networks.
If \(r\) is an inner reticulation 
vertex in a tree-child network without elementary
paths, then \(\outdeg(r)\ge 2\). Indeed, if \(\outdeg(r)=1\), then the unique child of
\(r\) would have to be a tree vertex by the tree-child property, and the corresponding
arc would be elementary.

Given a phylogenetic network \(\N\), we order tree-paths by directed subpath inclusion $\subseteq$. More precisely, a tree-path
\(P\) is contained in a tree-path \(P'\), in symbols $P \subseteq P'$, if \(P\) occurs as a directed subpath of
\(P'\). A tree-path \(P\) is \emph{maximal} if it is not properly contained in any
other tree-path of \(\N\).

The following lemma shows that tree-paths ending at the same vertex are linearly
ordered by directed subpath inclusion. 

\begin{lem}\label{thm:tree-path-contains} 
Let \(P:u\rightsquigarrow v\) and \(Q:w\rightsquigarrow v\) be two tree-paths of a
phylogenetic network \(\N\) ending at the same vertex \(v\). Then either \(P\) is
contained in \(Q\), or \(Q\) is contained in \(P\).
\end{lem}

\begin{proof}
Write
$P=(u_1,u_2,\ldots,u_p)$, and $Q=(w_1,w_2,\ldots,w_q)$,
where \(u_1=u\), \(w_1=w\), and \(u_p=w_q=v\). Since both \(P\) and \(Q\) are
tree-paths, every non-initial vertex of each path is a tree vertex. In particular,
each such vertex has a unique parent.

Starting from the common endpoint \(v\), we compare the two paths backwards. If
both previous vertices exist, then they are both parents of the same tree vertex,
and therefore they coincide. Repeating this argument backwards, the two paths
coincide until one of them reaches its first vertex. Hence one of the two paths is
a directed subpath of the other.
\end{proof}

As a direct consequence of Lemma \ref{thm:tree-path-contains}, we have the following result.

\begin{cor}\label{thm:tree-path-unicity} 
Let \(\N\) be a phylogenetic network and let
\(P:u\rightsquigarrow v\)
be a tree-path. Then \(P\) is the unique directed path from \(u\) to \(v\).
\end{cor}

\subsection{The $\mu$-representation\cite{cardona2008comparison}}


Let \(\N\) be a phylogenetic network on \([n]\). For every vertex \(v \in V(\N)\) we define the \emph{path multiplicity} of \(v\) to the leaf \(l_i\) labelled by \(i \in [n]\) as the number of paths \(m_i(v)\) from \(v\) to \(l_i\). The \emph{path-multiplicity vector}, or \(\mu\)-\emph{vector}, of \(v\) is defined as
\[\mu(v) := (m_1(v), m_2(v), \ldots, m_n(v)) \in \mathbb{N}^n.\]
In particular, for a leaf \(l_i\), we have
\[\mu(l_i) = \delta_i = (\underbrace{0, \ldots, 0, \operatornamewithlimits{1}^i, 0, \ldots, 0}_n),\]
with the non-zero entry at index \(i\).

The \(\mu\)-\emph{representation} of \(\N\) is the multiset
\[\mu(\N) = \{\mu(v): v \in V(\N)\}.\]

\begin{lem}[Lemma 4, \cite{cardona2008comparison}]\label{lem:children_mu}
    Let $N$ be a phylogenetic network and $v$ an inner vertex. Let $\{v_1,\ldots, v_k\}$ be the set of children of $v$. Then $\mu(v)=\mu(v_1)+\cdots + \mu(v_k)$.
\end{lem}

As a consequence of Lemma 5 in \cite{cardona2008comparison}, we have the following result. 


\begin{prop}\label{prop:mu_is_set}
    If $N$ is a tree-child network without elementary paths, then distinct vertices have distinct $\mu$-vectors. Consequently, its $\mu$-representation can be regarded as a set.
\end{prop}

Cardona et al. prove in \cite{cardona2008comparison} that \(\mu\)-representations can be used to encode tree-child networks up to isomorphism:
\begin{thm}\label{thm:mu:cardona-isomorphism}
	Let \(\N, \N'\) be tree-child phylogenetic networks on \([n]\). Then, \(\N\) and \(\N'\) are isomorphic if, and only if, \(\mu(\N) = \mu(\N')\).
\end{thm}


 Theorem \ref{thm:mu:cardona-isomorphism} is valid under the hypotheses of this paper. This means that it is valid on the set of non-binary tree-child networks without elementary paths and with arbitrary in-degree reticulation nodes.

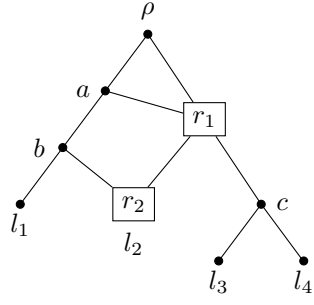
\begin{figure}
	\centering
	\begin{tikzpicture}[scale=.75]
		\filldraw [] (0, 0) circle (2pt);
		\node at (0, 0) [circle, above] (R) {\(\rho\)};
		\filldraw [] (-.75, -1) circle (2pt);
		\node at (-.75, -1) [circle, left] (A) {\(a\)};
		\filldraw [] (-1.5, -2) circle (2pt);
		\node at (-1.5, -2) [circle, left] (B) {\(b\)};
		\filldraw [] (-2.25, -3) circle (2pt);
		\node at (-2.25, -3) [below] (L1) {\(l_1\)};
		\draw [] (0, 0) -- (-2.25, -3);
		
		\node at (1, -1.5) [draw] (R1) {\(r_1\)};		\draw [] (0, 0) -- (R1) -- (-.75, -1);
		
		\node at (-.25, -3) [draw] (R2) {\(r_2\)};
		\draw [] (-1.5, -2) -- (R2) -- (R1);
		
		\node at (R2) [draw=none, below=7pt] () {\(l_2\)};
		
		\filldraw (2, -3) circle (2pt);
		\node at (2, -3) [circle, right] (C) {\(c\)};
		\filldraw (2.75, -4) circle (2pt);
		\node at (2.75, -4) [below] () {\(l_4\)};
		\filldraw (1.25, -4) circle (2pt);
		\node at (1.25, -4) [below] () {\(l_3\)};
		
		\draw [] (R1) -- (2, -3) -- (1.25, -4);
		\draw [] (2, -3) -- (2.74, -4);
	\end{tikzpicture}
	\caption{A phylogenetic tree-child network.}
	\label{fig:phylo-example}
\end{figure}

\begin{exm} 
Consider the tree-child network \(\N\) on $[4]$ represented in Figure~\ref{fig:phylo-example}, where \(l_i\) denotes the leaf labelled by \(i\). Then the \(\mu\)-vectors of each node are the following:

\begin{center}
	\begin{tabular}{|c|c|}
		\hline
		Node & \(\mu\)-vector \\ \hline
		\(\rho\) & \((1, 3, 2, 2)\) \\
		\(a\) & \((1, 2, 1, 1)\) \\
		\(b\) & \((1, 1, 0, 0)\) \\
		\(c\) & \((0, 0, 1, 1)\) \\
		\(r_1\) & \((0, 1, 1, 1)\) \\
		\hline
	\end{tabular}
	\hspace{1cm}
	\begin{tabular}{|c|c|}
		\hline
		Node & \(\mu\)-vector \\ \hline
		\(l_1\) & \((1, 0, 0, 0)\) \\
		\(l_2\)/\(r_2\) & \((0, 1, 0, 0)\) \\
		\(l_3\) & \((0, 0, 1, 0)\) \\
		\(l_4\) & \((0, 0, 0, 1)\) \\
		\hline
	\end{tabular}
\end{center}

Note that node \(l_2=r_2\) is both a leaf node and a reticulation. Then, the \(\mu\)-representation of \(\N\) can be identified with the set
\begin{equation*}
	\mu(\N) = \{\delta_1, \delta_2, \delta_3, \delta_4, (1, 1, 0, 0), (0, 0, 1, 1), (0, 1, 1, 1), (1, 2, 1, 1), (1, 3, 2, 2)\}.
\end{equation*}
\end{exm}

\section{Tree-paths on tree-child networks}

 If we focus on the space of tree-child networks, it is possible to determine when a tree-path is maximal.

\begin{lem}\label{thm:tree-path-maximal-condition}
A tree-path of a tree-child network \(\N\) is maximal if and only if it starts either
at the root or at a reticulation vertex, and it ends at a leaf.
\end{lem}

\begin{proof}
Let $P=(u_1,u_2,\ldots,u_k)$
be a tree-path of \(\N\).
Assume first that \(P\) is maximal. Suppose that \(u_1\) is neither the root nor a
reticulation vertex. Then \(u_1\) is a tree vertex, and hence \(\indeg(u_1)=1\).
Let \(u_0\) be its unique parent. Since the arc \((u_0,u_1)\) ends in a tree vertex,
the path $P'=(u_0,u_1,\ldots,u_k)$
is again a tree-path, properly containing \(P\). This contradicts the maximality of
\(P\). Therefore, \(u_1\) is either the root or a reticulation vertex.

Now suppose that \(u_k\) is not a leaf. Since \(\N\) is tree-child, \(u_k\) has at
least one child \(u_{k+1}\) that is a tree vertex. Hence
$P'=(u_1,\ldots,u_k,u_{k+1})$
is a tree-path properly containing \(P\), again contradicting the maximality of
\(P\). Thus \(u_k\) must be a leaf.

Conversely, assume that \(P\) starts either at the root or at a reticulation vertex,
and that it ends at a leaf. Suppose, for contradiction, that \(P\) is not maximal.
Then there exists a tree-path \(P'\) properly containing \(P\) as a directed subpath.
Since \(u_k\) is a leaf, \(P'\) cannot extend \(P\) after \(u_k\). Therefore \(P'\)
must extend \(P\) before \(u_1\). In particular, there is a vertex \(u_0\notin P\)
such that \((u_0,u_1)\) is an arc of \(P'\).
But \(P'\) is a tree-path, so every arc of \(P'\) ends in a tree vertex. Hence
\(u_1\) must be a tree vertex. This is impossible if \(u_1\) is a reticulation
vertex. It is also impossible if \(u_1=\rho\), since the root has no parent.
Therefore no such \(P'\) exists, and \(P\) is maximal.
\end{proof}

\begin{cor}\label{cor:unique-maximal-tree-path-leaf}
Let \(\N\) be a tree-child network. Then each leaf \(l \in L(\N)\) is contained in a
unique maximal tree-path. 
\end{cor}

\begin{proof}
First, there exists at least one maximal tree-path containing \(l\). Indeed, the
trivial path \((l)\) is a tree-path, and since \(\N\) is finite, it is contained in
some maximal tree-path.

We now prove uniqueness. Let \(T\) and \(T'\) be two maximal tree-paths containing
\(l\). Since \(l\) is a leaf, both \(T\) and \(T'\) end at \(l\). Therefore, by
Lemma~\ref{thm:tree-path-contains}, either \(T\) is contained in \(T'\), or \(T'\)
is contained in \(T\). Since both tree-paths are maximal, the containment must be
equality. Thus \(T=T'\).
\end{proof}

Let \(\N\) be a tree-child network. By Corollary~\ref{cor:unique-maximal-tree-path-leaf}, the following notation is
well-defined. For each leaf \(l\in L(\N)\), we denote by $T(l)=(u_1,u_2,\ldots,u_m=l)$
the unique maximal tree-path ending at \(l\). When no confusion is possible, we
identify \(T(l)\) with its set of vertices.

We now consider the map $r:L(\N)\longrightarrow R(\N)\cup\{\rho\}$, that sends a leaf $l$ to the first node in $T(l)$. By Lemma~\ref{thm:tree-path-maximal-condition} and Corollary~\ref{cor:unique-maximal-tree-path-leaf} the map $r$ is well-defined. That is, $r(l)$ lies in \(R(\N)\cup\{\rho\}\) and $T(l)$ is unique; then its first node too.
\begin{prop}
   Let \(\N\) be a tree-child network.  The map \(r\) is surjective.
\end{prop}
\begin{proof}
Let \(u\in R(\N)\cup\{\rho\}\). If \(u\) is a
leaf, then the trivial path \((u)\) is a maximal tree-path and \(u=r(u)\). If \(u\)
is not a leaf, then, since \(\N\) is tree-child, we can start at \(u\) and repeatedly
choose a tree child until we reach a leaf \(l\). This gives a tree-path from \(u\)
to \(l\). By Lemma~\ref{thm:tree-path-maximal-condition}, this tree-path is maximal,
and therefore \(u=r(l)\).
\end{proof}

\begin{cor}\label{thm:tp:col:v-to-l-then-v-to-rl}
Let \(\N\) be a tree-child network. For each leaf \(l\in L(\N)\) and each vertex
\(v\in V(\N)\), there exists a path from \(v\) to \(l\) if and only if either
\(v\in T(l)\) or \(v > r(l)\). 
\end{cor}

\begin{proof}
If \(v\in T(l)\), then there is a path from \(v\) to \(l\) along \(T(l)\). If
\(v > r(l)\), then concatenating a path from \(v\) to \(r(l)\) with the tree-path
\(T(l)\) gives a path from \(v\) to \(l\).

Conversely, suppose that there exists a path
$
P=(v=x_0,x_1,\ldots,x_s=l)
$
from \(v\) to \(l\). If \(v\in T(l)\), there is nothing to prove. Assume therefore
that \(v\notin T(l)\). 
Let \(x_q\) be the first vertex of \(P\) that belongs to \(T(l)\). Such a vertex
exists because \(l\in T(l)\). Since \(v\notin T(l)\), we have \(q>0\). We claim that
\(x_q=r(l)\). Indeed, if \(x_q\ne r(l)\), then \(x_q\) is a non-initial vertex of
the tree-path \(T(l)\), and hence \(x_q\) is a tree vertex. Therefore \(x_q\) has a
unique parent. But \(x_{q-1}\) is a parent of \(x_q\), and the predecessor of
\(x_q\) in \(T(l)\) is another parent of \(x_q\) belonging to \(T(l)\), contradicting
the choice of \(x_q\) as the first vertex of \(P\) in \(T(l)\). Hence \(x_q=r(l)\),
and so \(v > r(l)\).
\end{proof}

Once the map $r$  has been defined, the vertices of the maximal tree-paths can be described purely in terms of reachability and the relative position of the vertices \(r(l)\). The following result should therefore be understood as a description of \(T(l)\) relative to the map \(r\).

\begin{prop}\label{thm:tp:prop:max-tree-path-conditions-accessibility}
Let \(\N\) be a tree-child network. For each leaf \(l\in L(\N)\), the maximal
tree-path \(T(l)\), identified with its set of vertices, is given by
\[
T(l)
=
\{v\in V(\N):
v\rightsquigarrow l
\text{ and }
v\not\rightsquigarrow l'
\text{ for every } l'\in L(\N) \text{ such that } r(l')>r(l)\}.
\]
\end{prop}

\begin{proof}
Let
$
A_l=
\{v\in V(\N):
v\rightsquigarrow l
\text{ and }
v\not\rightsquigarrow l'
\text{ for every } l'\in L(\N) \text{ such that } r(l')>r(l)\}$.

We prove that \(T(l)=A_l\). First, let \(v\in T(l)\). Then clearly \(v\rightsquigarrow l\). Suppose, for a
contradiction, that there exists a leaf \(l'\in L(\N)\) such that
$
r(l')>r(l)
$
and
$
v\rightsquigarrow l'$.
By Corollary~\ref{thm:tp:col:v-to-l-then-v-to-rl}, applied to \(l'\), either
\(v\in T(l')\) or \(v>r(l')\). 
The second possibility cannot occur. Indeed, since \(v\in T(l)\), we have
\(r(l)\ge v\). If \(v>r(l')\), then
$r(l)\ge v>r(l')>r(l)$,
which would give a directed cycle.
Hence \(v\in T(l')\). Consider the subpaths of \(T(l)\) and \(T(l')\) ending at
\(v\):
$r(l)\rightsquigarrow v$ and $r(l')\rightsquigarrow v$.
Both are tree-paths. By Lemma~\ref{thm:tree-path-contains}, one is contained in the
other. Since \(r(l')>r(l)\), the subpath from \(r(l)\) to \(v\) is contained in the
subpath from \(r(l')\) to \(v\). Thus \(r(l)\) is a non-initial vertex of \(T(l')\).
But \(r(l)\) is a reticulation vertex, since \(r(l')>r(l)\) implies \(r(l)\ne\rho\).
This contradicts the fact that all non-initial vertices of a tree-path are tree
vertices. Therefore \(v\in A_l\).

Conversely, let \(v\in A_l\). Since \(v\rightsquigarrow l\), Corollary~\ref{thm:tp:col:v-to-l-then-v-to-rl}
implies that either \(v\in T(l)\) or \(v>r(l)\). If \(v\in T(l)\), there is
nothing to prove. Assume, for a contradiction, that \(v\notin T(l)\).
Since \(\N\) is tree-child, starting from \(v\) and repeatedly choosing a tree child,
we obtain a tree-path from \(v\) to some leaf \(l'\). Hence \(v\rightsquigarrow l'\).
Moreover, the maximal tree-path \(T(l')\) containing this tree-path starts at some
vertex \(r(l')\) satisfying $r(l')\ge v$.
Therefore
$r(l')\ge v>r(l)$,
and so \(r(l')>r(l)\). This contradicts the definition of \(A_l\), because
\(v\rightsquigarrow l'\). Hence \(v\in T(l)\).
\end{proof}

\begin{prop}\label{prop:cover_Tl}
Let \(\N\) be a tree-child network with leaf set
$L(\N)=\{l_1,\ldots,l_n\}$.
Then
$\mathcal T_\N=\{T(l_i):i=1,\ldots,n\}
$
is a cover of \(V(\N)\).
\end{prop}

\begin{proof}
Let \(v\in V(\N)\). If \(v\) is a leaf, then \(v\in T(v)\).
Assume now that \(v\) is not a leaf. Since \(\N\) is tree-child, we can start at
\(v\) and repeatedly choose a tree child until we reach a leaf \(l\), and it yields a tree-path from \(v\) to
\(l\).
By Corollary~\ref{cor:unique-maximal-tree-path-leaf}, there is a unique maximal
tree-path ending at \(l\), namely \(T(l)\). Therefore the tree-path from \(v\) to
\(l\) is contained in \(T(l)\), and hence
$v\in T(l)$.
Thus every vertex of \(\N\) belongs to at least one member of
\(\mathcal T_\N\), and so \(\mathcal T_\N\) is a cover of \(V(\N)\).
\end{proof}

See Example \ref{exm:ordered-tree-path-example} for a decomposition of a tree-child network into maximal tree-paths. In particular, we can observe in this example that $T(l_i)$ and $T(l_j)$ have vertices in common although $l_i \neq l_j$; in other terms, the set $\mathcal T_\N=\{T(l_i):i=1,\ldots,n\}$ is a cover but not a partition of $V(N)$.

We now turn the cover by maximal tree-paths into a partition. Let \(\prec\) be a
total order on \([n]\). For every \(i\in [n]\), define
\[
P_\prec(l_i)
=
\{v \in V(\N): v \rightsquigarrow l_i \text{ and } v \not \rightsquigarrow l_j \text{ for every } j \prec i\}.
\]
Thus, each vertex \(v\) is assigned to the first label \(i\), with respect to
\(\prec\), such that \(l_i\) is a descendant of \(v\).

\begin{defn}\label{def:tree-path-order}
    Let \(\N\) be a tree-child network. A total ordering \(\prec\) of \([n]\) is a \emph{tree-path order} if every pair of indices \(i, j \in [n]\) such that \(r(l_i) > r(l_j)\) satisfies \(i \prec j\).
\end{defn}

\begin{prop}\label{prop:tree-path-order-partition}
    Let \(\N\) be a tree-child network, and let \(\prec\) be a tree-path order on \([n]\). Then, for every \(i \in [n]\), we have that
    \[P_\prec(l_i) = T(l_i) \setminus \bigcup_{j \prec i} T(l_j).\]
\end{prop}

\begin{proof}
    We will first prove that for every \(i \in [n]\), \(P_\prec(l_i) \subseteq T(l_i)\). Then, we will prove that, for any vertex \(v\), if \(v \in T(l_i) \cap T(l_j)\) with \(j \prec i\), then \(v \notin P_\prec(l_i)\).

    Let \(v\in P_\prec(l_i)\). Then \(v \rightsquigarrow l_i\), and \(v \not \rightsquigarrow l_j\) for every \(j \prec i\). Since \(\prec\) is a tree-path order, if \(r(l_j)>r(l_i)\), then \(j\prec i\). Therefore \(v \not \rightsquigarrow l_j\) for every \(l_j \in L(N)\) such that \(r(l_j) > r(l_i)\). By Proposition \ref{thm:tp:prop:max-tree-path-conditions-accessibility}, \(v \in T(l_i)\).

    For the second property, if \(v \in T(l_i) \cap T(l_j)\) with \(j \prec i\), then \(v \rightsquigarrow l_j\). Therefore, \(v \notin P_\prec(l_i)\).

    We have proved that \(P_\prec(l_i) \subseteq T(l_i) \setminus \bigcup_{j \prec i} T(l_j)\). For the other inclusion, let \(v \in T(l_i) \setminus \bigcup_{j \prec i} T(l_j)\). Since \(v \in T(l_i)\), we have that \(v \rightsquigarrow l_i\) and \(v \le r(l_i)\). Suppose that there exists an index \(j \prec i\) such that \(v \rightsquigarrow l_{j}\). We know that \(v \notin T(l_{j})\). By Corollary \ref{thm:tp:col:v-to-l-then-v-to-rl}, we have that \(r(l_{j}) < v \le r(l_i)\). Hence, it follows that \(i \prec j\), leading to a contradiction with the supposition that \(j \prec i\). Therefore, the other inclusion also holds.

    Finally, we have that \(P_\prec(l_i)\) corresponds to the set of vertices \(v\) such that \(v \in T(l_i)\) and \(v \notin T(l_j)\) for every \(j \prec i\).
\end{proof}

Proposition \ref{prop:tree-path-order-partition} states that, for tree-path orders, each vertex of the network is assigned to the first maximal tree-path, with respect to \(\prec\), that contains it. Equivalently, under the conditions of Proposition \ref{prop:tree-path-order-partition}, \(P_\prec(l_i)\) can be rewritten as \(P_\prec(l_i) = \{v \in T(l_i): v \notin T(l_j) \text{ for every } j \prec i\}\).

\begin{prop}\label{prop:partition_Pl}
Let \(\N\) be a tree-child network, and let \(\prec\) be a tree-path order on \([n]\).
Then
\[
\mathcal P_{\N,\prec}
=
\{P_\prec(l_i): i \in [n]\}
\]
is a partition of \(V(\N)\). Moreover,  \(P_\prec(l_i)\) is a directed
path ending at \(l_i\).
\end{prop}

\begin{proof}
By Proposition~\ref{prop:cover_Tl}, the collection
$\mathcal T_\N=\{T(l_i):l_i\in L(\N)\}$
is a cover of \(V(\N)\). Therefore every vertex of \(\N\) belongs to at least one
maximal tree-path. By Proposition \ref{prop:tree-path-order-partition}, since \(\prec\) is a total order on \([n]\), each vertex \(v\) is
assigned to the first index \(i\) with respect to \(\prec\), such that \(v \in T(l_i)\). Hence the sets
\(P_\prec(l_i)\) are pairwise disjoint and their union is \(V(\N)\). Thus
\(\mathcal P_{\N,\prec}\) is a partition of \(V(\N)\).

It remains to prove that each \(P_\prec(l_i)\) is a directed path ending at
\(l_i\). Since \(P_\prec(l_i)\subseteq T(l_i)\), let \(v\in P_\prec(l_i)\) and let
\(w\) be a vertex below \(v\) in \(T(l_i)\). Then \(w\rightsquigarrow l_i\). If
\(w\rightsquigarrow l_j\) for some \(j\prec i\), then
\(v\rightsquigarrow w\rightsquigarrow l_j\), contradicting
\(v\in P_\prec(l_i)\). Hence \(w\in P_\prec(l_i)\). Therefore
\(P_\prec(l_i)\) is a terminal subpath of \(T(l_i)\), and so it is a directed path
ending at \(l_i\).

\end{proof}

\begin{defn}\label{def:ordered-tree-path-decomposition}
Let \(\N\) be a phylogenetic network with leaf set
$L(\N)=\{l_1,\ldots,l_n\}$.
An \emph{ordered tree-path decomposition} of \(\N\) is a pair $(\prec,\mathcal P_\N)$,
where \(\prec\) is a total order on \([n]\) and
$\mathcal P_\N=\{P(l_i):i \in [n]\}$
is a partition of \(V(\N)\), satisfying the following conditions.

\begin{enumerate}
    \item For every \(i\in[n]\), \(P(l_i)\) can be ordered as a directed
    path ending at \(l_i\),
    \[
    P(l_i)=
    \bigl(
    v_i^{m_i},
    v_i^{m_i-1},
    \ldots,
    v_i^2,
    v_i^1=l_i
    \bigr).
    \]

    \item Every arc of \(\N\) is either an internal arc of one of these paths, $(v_i^j,v_i^{j-1})$,
    for $i\in[n]$, and $j=2,\ldots,m_i$,
    or an arc from a vertex of one path to the maximal vertex of a later path,    $(v_i^j,v_k^{m_k})$, with $i\prec k$.
\end{enumerate}
\end{defn}

See Example \ref{exm:ordered-tree-path-example} for ordered tree-path decompositions of a tree-child network.

\begin{prop}\label{prop:ordering-is-decomposition}
    Let \(\N\) be a tree-child network and let \(\prec\) be a tree-path order of \([n]\). Then,
    \((\prec, \mathcal{P}_{\N, \prec})\)
    is an ordered tree-path decomposition of \(\N\).
\end{prop}

\begin{proof}
    By Proposition \ref{prop:partition_Pl}, \(\mathcal{P}_{\N, \prec}\) is a partition of \(V(\N)\), and each \(P_\prec(l_i)\) is a directed path ending at \(l_i\). Hence condition \((1)\) of Definition \ref{def:ordered-tree-path-decomposition} is satisfied.

    Write
    $    P_\prec(l_i)=(v_i^{m_i},\ldots,v_i^1=l_i).$
    By Proposition \ref{prop:tree-path-order-partition}, \(P_\prec(l_i)\subseteq T(l_i)\), and therefore every vertex \(v_i^j\) with \(j<m_i\) is a tree vertex. Hence the only arc entering \(v_i^j\), for \(j<m_i\), is the internal arc \((v_i^{j+1},v_i^j)\). Therefore, every arc entering \(P_\prec(l_i)\) which is not an internal arc must end at its maximal vertex \(v_i^{m_i}\).

    Let \((u,v_i^{m_i})\) be such an arc. Since \(\N\) is acyclic, \(u\notin P_\prec(l_i)\). As \(\mathcal{P}_{\N, \prec}\) is a partition of \(V(\N)\), we have \(u\in P_\prec(l_k)\) for some \(k\ne i\). If \(i\prec k\), then \(u\rightsquigarrow l_i\), and the definition of \(P_\prec(l_k)\) would imply \(u\notin P_\prec(l_k)\), a contradiction. Hence \(k\prec i\). Thus every non-internal arc goes from a vertex of an earlier path to the maximal vertex of a later path, and condition \((2)\) of Definition \ref{def:ordered-tree-path-decomposition} is satisfied.

    Therefore, \((\prec, \mathcal{P}_{\N, \prec})\) is an ordered tree-path decomposition of \(\N\).
\end{proof}

\begin{thm}
\label{thm:ordered-tree-path-characterization}
A phylogenetic network \(\N\) is tree-child if and only if it admits an ordered
tree-path decomposition.
\end{thm}

\begin{proof}

Assume first that \(\N\) is tree-child. Since \(\N\) is acyclic, the relation on
\([n]\) given by \(r(l_i)>r(l_j)\) is acyclic and can be extended to a total order
\(\prec\) on \([n]\). By Definition~\ref{def:tree-path-order}, \(\prec\) is a tree-path
order. Hence, by Proposition~\ref{prop:ordering-is-decomposition},
\((\prec,\mathcal P_{\N,\prec})\) is an ordered tree-path decomposition of \(\N\).

Conversely, suppose that \(\N\) admits an ordered tree-path decomposition
$(\prec,\mathcal P_\N)$. 
Let \(u\) be an inner vertex of \(\N\). Since \(\mathcal P_\N\) is a partition of
\(V(\N)\), there exists a unique \(i\) such that \(u\in P(l_i)\). Write
$P(l_i)=
\bigl(
v_i^{m_i},
v_i^{m_i-1},
\ldots,
v_i^2,
v_i^1=l_i
\bigr)$.
Since \(u\) is not a leaf, we have
$u=v_i^j$
for some \(j>1\). Since \(P(l_i)\) is a directed path, the arc
$(v_i^j,v_i^{j-1})$
belongs to \(E(\N)\).

We claim that \(v_i^{j-1}\) is a tree vertex. By condition \(2\) in
Definition~\ref{def:ordered-tree-path-decomposition}, every arc entering a
non-maximal vertex of a path must be the corresponding internal arc of that path.
Since \(v_i^{j-1}\) is not maximal in \(P(l_i)\), the only arc entering it is
$(v_i^j,v_i^{j-1})$. Therefore
$\indeg(v_i^{j-1})=1$, and so \(v_i^{j-1}\) is a tree vertex.
Thus every inner vertex \(u\) has a tree vertex as a child. Hence \(\N\) is
tree-child.
\end{proof}

\begin{rmk}
Notice that an ordered tree-path decomposition contains enough information to recover the
network once the \emph{bridges} (the arcs from a node of one path to a maximal node of another path) are specified. 
\end{rmk}

\begin{rmk}
The previous characterization should not be understood as a tree-child test for a phylogenetic network: this can already be checked locally by
verifying that every inner vertex has at least one tree child. Instead, the theorem
provides a structural certificate for the tree-child property, decomposing the
network into directed paths and bridges. This is the feature that will be used
later to express the tree-child condition in terms of \(\mu\)-representations.
\end{rmk}

\begin{exm}\label{exm:ordered-tree-path-example}

Consider the tree-child phylogenetic network \(\N\) on \([3]\),
depicted in
Figure~\ref{fig:ordered-tree-path-example-original}. 

The maximal tree-paths of \(\N\) are
$T(l_2)=(\rho,a,b,l_2)$, $T(l_3)=(\rho,a,l_3)$, and $T(l_1)=(l_1)$.
Hence
$\mathcal T_\N=\{T(l_1),T(l_2),T(l_3)\}$
is a cover of \(V(\N)\), but not a partition, because the vertices \(\rho\) and
\(a\) belong to both \(T(l_2)\) and \(T(l_3)\).

We now consider the induced partitions for some tree-path orders on \(\{1, 2, 3\}\).
If we consider the order \(2\prec 3\prec 1\), then \(P_\prec(l_2)=(\rho,a,b,l_2)\),
\(P_\prec(l_3)=(l_3)\), and \(P_\prec(l_1)=(l_1)\); the bridges are
$(a,l_3)$, $(b,l_1)$, and $(\rho,l_1)$.
They all satisfy the condition in
Definition~\ref{def:ordered-tree-path-decomposition}: the arc \((a,l_3)\) goes from
a vertex of \(P_\prec(l_2)\) to the maximal vertex of the later path
\(P_\prec(l_3)\), while \((b,l_1)\) and \((\rho,l_1)\) go to the maximal vertex of
the later path \(P_\prec(l_1)\). This ordered tree-path decomposition is represented in
Figure~\ref{fig:ordered-tree-path-example-231}.

If we now consider the order \(3\prec 2\prec 1\), then \(P_\prec(l_3)=(\rho,a,l_3)\),
\(P_\prec(l_2)=(b,l_2)\), and \(P_\prec(l_1)=(l_1)\).
The bridges are now
$(a,b)$, $(b,l_1)$, and $(\rho,l_1)$.
Again, this also yields an ordered tree-path decomposition; see
Figure~\ref{fig:ordered-tree-path-example-321}.

Notice that these are the only two tree-path orders. Indeed,
\(r(l_2)=r(l_3)=\rho\) and \(r(l_1)=l_1\). Hence
\(r(l_2)>r(l_1)\) and \(r(l_3)>r(l_1)\), so Definition~\ref{def:tree-path-order}
requires both \(2\prec1\) and \(3\prec1\). Thus index \(1\) must be last, while
indices \(2\) and \(3\) can occur in either order. 

Finally, this example illustrates that, once the directed paths of the partition and
the corresponding bridges are known, the original network \(\N\) is recovered
uniquely by taking the internal arcs of the paths together with those bridges.
\end{exm}

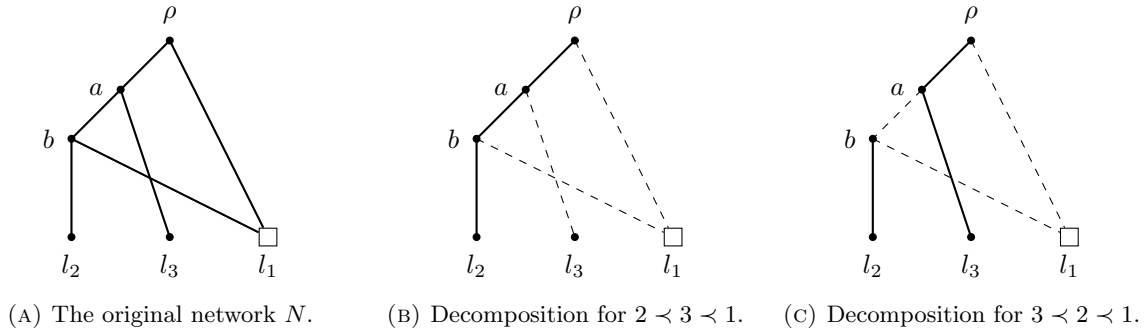
\begin{figure}[ht]
\centering

\begin{subfigure}{.32\textwidth} 
\centering
\begin{tikzpicture}[scale=.65]
    \filldraw [] (0,0) circle (2pt);
    \node at (0,0) [above=3pt] {\(\rho\)};

    \filldraw [] (-1,-1) circle (2pt);
    \node at (-1,-1) [left=3pt] {\(a\)};

    \filldraw [] (-2,-2) circle (2pt);
    \node at (-2,-2) [left=3pt] {\(b\)};

    \node at (2,-4) [draw] (L1) {};
    \node at (2,-4) [below=3pt] {\(l_1\)};

    \filldraw [] (-2,-4) circle (2pt);
    \node at (-2,-4) [below=3pt] {\(l_2\)};

    \filldraw [] (0,-4) circle (2pt);
    \node at (0,-4) [below=3pt] {\(l_3\)};

    \draw [thick] (0,0) -- (-1,-1);
    \draw [thick] (0,0) -- (L1);
    \draw [thick] (-1,-1) -- (-2,-2);
    \draw [thick] (-1,-1) -- (0,-4);
    \draw [thick] (-2,-2) -- (L1);
    \draw [thick] (-2,-2) -- (-2,-4);
\end{tikzpicture}

\caption{The original network \(\N\).}
\label{fig:ordered-tree-path-example-original}
\end{subfigure} 
\hfill
\begin{subfigure}{.32\textwidth}
\centering
\begin{tikzpicture}[scale=.65]
    \filldraw [] (0,0) circle (2pt);
    \node at (0,0) [above=3pt] {\(\rho\)};

    \filldraw [] (-1,-1) circle (2pt);
    \node at (-1,-1) [left=3pt] {\(a\)};

    \filldraw [] (-2,-2) circle (2pt);
    \node at (-2,-2) [left=3pt] {\(b\)};

    \node at (2,-4) [draw] (L1) {};
    \node at (2,-4) [below=3pt] {\(l_1\)};

    \filldraw [] (-2,-4) circle (2pt);
    \node at (-2,-4) [below=3pt] {\(l_2\)};

    \filldraw [] (0,-4) circle (2pt);
    \node at (0,-4) [below=3pt] {\(l_3\)};

    \draw [thick] (0,0) -- (-1,-1) -- (-2,-2) -- (-2,-4);

    \draw [dashed] (-1,-1) -- (0,-4);
    \draw [dashed] (-2,-2) -- (L1);
    \draw [dashed] (0,0) -- (L1);
\end{tikzpicture}

\caption{Decomposition for \(2\prec 3\prec 1\).}
\label{fig:ordered-tree-path-example-231}
\end{subfigure}
\hfill
\begin{subfigure}{.32\textwidth}
\centering
\begin{tikzpicture}[scale=.65]
    \filldraw [] (0,0) circle (2pt);
    \node at (0,0) [above=3pt] {\(\rho\)};

    \filldraw [] (-1,-1) circle (2pt);
    \node at (-1,-1) [left=3pt] {\(a\)};

    \filldraw [] (-2,-2) circle (2pt);
    \node at (-2,-2) [left=3pt] {\(b\)};

    \node at (2,-4) [draw] (L1) {};
    \node at (2,-4) [below=3pt] {\(l_1\)};

    \filldraw [] (-2,-4) circle (2pt);
    \node at (-2,-4) [below=3pt] {\(l_2\)};

    \filldraw [] (0,-4) circle (2pt);
    \node at (0,-4) [below=3pt] {\(l_3\)};

    \draw [thick] (0,0) -- (-1,-1) -- (0,-4);
    \draw [thick] (-2,-2) -- (-2,-4);

    \draw [dashed] (-1,-1) -- (-2,-2);
    \draw [dashed] (-2,-2) -- (L1);
    \draw [dashed] (0,0) -- (L1);
\end{tikzpicture}

\caption{Decomposition for \(3\prec 2\prec 1\).}
\label{fig:ordered-tree-path-example-321}
\end{subfigure}

\caption{A tree-child network, panel (\textsc{a}), with a reticulation leaf, together with its two
ordered tree-path decompositions, panels (\textsc{b}) and (\textsc{c}). Solid arcs represent the internal
arcs of the paths in the partition, whereas dashed arcs are the bridges between
distinct paths.}
\label{fig:ordered-tree-path-example}
\end{figure}

\section{Tree-child \texorpdfstring{\(\mu\)}{mu}-representations}
\label{sec:tc-mu-representations}

In the previous section we characterized tree-child networks in terms of ordered
tree-path decompositions. We now translate this characterization into the language
of \(\mu\)-representations. Our goal is to characterize those finite sets of vectors
$M\subseteq \mathbb N^n$
which arise as the \(\mu\)-representation of a tree-child phylogenetic network on
\([n]\) without elementary paths.
Throughout this section, for networks without elementary paths, \(\mu(\N)\) will be identified with the corresponding set of vectors, as in Proposition~\ref{prop:mu_is_set}. 

 We use the coordinatewise
partial order on \(\mathbb N^n\): for \(\mu,\nu\in\mathbb N^n\), we write
$\mu\leq \nu$
if
$\mu^{(i)}\leq \nu^{(i)}$ for every 
$i\in\{1,\ldots,n\}$.
We write \(\mu<\nu\) if \(\mu\leq \nu\) and \(\mu\neq\nu\).

Let \(M\subseteq\mathbb N^n\) be a finite set with a maximum element for the
coordinatewise order. We denote this maximum element by $\mu_\rho$.
Thus
$\mu\leq \mu_\rho$ for  every $\mu\in M$.
The vector \(\mu_\rho\) induces a canonical total order on the coordinates. We define
\[
i\prec_M j \text{ if and only if } 
\mu_\rho^{(i)}<\mu_\rho^{(j)}
\text{ or }
\mu_\rho^{(i)}=\mu_\rho^{(j)}
\text{ and }
i<j.
\]
Thus coordinates are ordered by their value in the maximal vector, with the
natural order on \([n]\) used only to break ties.

Given a finite set \(M\subseteq\mathbb N^n\) with maximum element \(\mu_\rho\), we
use the canonical order \(\prec_M\) to partition \(M\) according to the first
non-zero coordinate of each vector.
For each \(i\in[n]\), define
\[
P_i(M)
=
\{
\mu\in M:
\mu^{(i)}>0
\text{ and }
\mu^{(h)}=0
\text{ for every }h\prec_M i
\}.
\]
Equivalently, \(P_i(M)\) consists of those vectors whose first non-zero coordinate,
with respect to the order \(\prec_M\), is the \(i\)-th coordinate.

\begin{lem}\label{lem:mu-is-tree-child-order}
Let \(\N\) be a tree-child network without elementary paths and let \(M := \mu(\N)\) be its \(\mu\)-representation. Then, \(\prec_M\) is a tree-path order on \([n]\).
\end{lem}

\begin{proof}
    We prove that if \(r(l_i)>r(l_j)\) for two indices \(i,j\in[n]\), then \(\mu_\rho^{(i)}<\mu_\rho^{(j)}\), and therefore \(i\prec_M j\).

    Let \(\rho\) be the root of \(\N\) and let \(\mu_\rho:=\mu(\rho)\). For every vertex \(v\in V(\N)\), every directed path from \(v\) to a leaf can be extended backwards to a directed path from \(\rho\) to the same leaf. Hence \(\mu(v)\leq\mu_\rho\), and therefore \(\mu_\rho\) is the maximum element of \(M\).

    Let \(p(u,v)\) denote the number of directed paths from \(u\) to \(v\). By Corollary \ref{thm:tree-path-unicity}, there is a unique directed path from \(r(l_i)\) to \(l_i\). Moreover, every path from \(\rho\) to \(l_i\) contains \(r(l_i)\). Therefore
    \[
    p(\rho,r(l_i))=p(\rho,l_i)=\mu_\rho^{(i)}.
    \]
    Analogously, \(p(\rho,r(l_j))=\mu_\rho^{(j)}\). Since \(r(l_i)>r(l_j)\), fix a directed path from \(r(l_i)\) to \(r(l_j)\), and let \(e\) be its last edge. Concatenating this path with each path from \(\rho\) to \(r(l_i)\) gives at least \(p(\rho,r(l_i))\) distinct paths from \(\rho\) to \(r(l_j)\) ending with \(e\). Since \(r(l_j)\) is a reticulation vertex, it has another incoming edge \(e'\ne e\), which yields at least one additional path from \(\rho\) to \(r(l_j)\). Hence
    \[
    p(\rho,r(l_j))>p(\rho,r(l_i)),
    \]
    and thus \(\mu_\rho^{(j)}>\mu_\rho^{(i)}\). Therefore \(i\prec_M j\), and \(\prec_M\) is a tree-path order on \([n]\).
\end{proof}


\begin{lem}\label{lem:mu-partition}
Let \(M\subseteq\mathbb N^n\) be a finite set of non-zero vectors with maximum
element \(\mu_\rho\). Then
$\mathcal P_M
=
\{P_i(M):i \in [n]\}$
is a partition of \(M\).
\end{lem}

\begin{proof}
Let \(\mu\in M\). Since \(\mu\neq 0\), there exists at least one index \(i\) such
that \(\mu^{(i)}>0\). Since \(\prec_M\) is a total order, there is a unique minimal
such index with respect to \(\prec_M\). By definition, \(\mu\) belongs to the
corresponding set \(P_i(M)\). Hence the sets \(P_i(M)\) cover \(M\).

Moreover, \(\mu\) cannot belong to two distinct sets \(P_i(M)\) and \(P_j(M)\),
because this would mean that both \(i\) and \(j\) are the first non-zero coordinate
of \(\mu\) with respect to \(\prec_M\). Therefore the sets \(P_i(M)\) are pairwise
disjoint, and so they form a partition of \(M\).
\end{proof}

\begin{rmk}
Notice that Lemma \ref{lem:mu-partition} is the \(\mu\)-vector analogue of Proposition \ref{prop:partition_Pl} used in the previous
section. 
\end{rmk}

We next relate the partition of \(M\) with the ordered tree-path decompositions introduced in 
the previous section. For this purpose, we first state the correspondence for an
arbitrary total order \(\prec\) on \([n]\).
Let
\[
P_i^\prec(M)
=
\{
\mu\in M:
\mu^{(i)}>0
\text{ and }
\mu^{(h)}=0
\text{ for every }h\prec i
\}.
\]
Thus \(P_i^\prec(M)\) consists of those vectors whose first non-zero coordinate,
with respect to \(\prec\), is the \(i\)-th coordinate.

\begin{lem}\label{lem:reachable-leaves-from-ordered-path}
Let \(\N\) be a phylogenetic network admitting an ordered tree-path decomposition
$(\prec,\mathcal P_\N)$,
with $\mathcal P_\N=\{P(l_i):i \in [n]\}$.
If \(v\in P(l_i)\), then every leaf reachable from \(v\) is of the form \(l_k\) with
$i\preceq k$.
In particular, \(v\rightsquigarrow l_i\), and no leaf \(l_h\) with \(h\prec i\) is
reachable from \(v\).
\end{lem}

\begin{proof}
Write
$P(l_i)=
\bigl(
v_i^{m_i},
v_i^{m_i-1},
\ldots,
v_i^1=l_i
\bigr)$.
Since \(v\in P(l_i)\), the internal arcs of \(P(l_i)\) give a directed path from
\(v\) to \(l_i\). Hence \(v\rightsquigarrow l_i\).

Now consider any directed path starting at \(v\). By
Definition~\ref{def:ordered-tree-path-decomposition}, every arc of \(\N\) is either
an internal arc of one of the paths in \(\mathcal P_\N\), or a bridge from a vertex
of one path to the maximal vertex of a later path. Therefore, along any directed
path, the current path can only stay the same or move to a later path with respect
to \(\prec\). Hence a path starting in \(P(l_i)\) cannot reach a leaf \(l_h\) with
\(h\prec i\). This proves the claim.
\end{proof}

\begin{prop}\label{prop:mu-partition-from-ordered-decomposition}
Let \(\N\) be a tree-child phylogenetic network without elementary paths with leaf set $L(\N)=\{l_1,\ldots,l_n\}$,
and suppose that \(\N\) admits an ordered tree-path decomposition
$(\prec,\mathcal P_\N = \{P(l_i):i\in[n]\})$.
Let $M=\mu(\N)$. 
Then, for every \(i\in[n]\),
\[
\mu(P(l_i)):=\{\mu(v):v\in P(l_i)\}=P_i^\prec(M).
\]
\end{prop}

\begin{proof}
Let \(v\in P(l_i)\). By Lemma~\ref{lem:reachable-leaves-from-ordered-path}, we have
\(v\rightsquigarrow l_i\), and no leaf \(l_h\) with \(h\prec i\) is reachable from
\(v\). Since the \(h\)-th coordinate of \(\mu(v)\) counts the number of directed
paths from \(v\) to \(l_h\), it follows that
$\mu(v)^{(i)}>0$ 
and
$\mu(v)^{(h)}=0$ for every $h\prec i$.
Thus \(\mu(v)\in P_i^\prec(M)\), and hence
$\mu(P(l_i))\subseteq P_i^\prec(M)$.

Conversely, let \(\mu\in P_i^\prec(M)\). Since \(M=\mu(\N)\), there exists a vertex
\(v\in V(\N)\) such that $\mu=\mu(v)$.
As \(\mathcal P_\N\) is a partition of \(V(\N)\), there is a unique
\(k\in[n]\) such that
$v\in P(l_k)$.
Applying Lemma~\ref{lem:reachable-leaves-from-ordered-path} to \(P(l_k)\), we get
that \(v\rightsquigarrow l_k\), and that no leaf \(l_h\) with \(h\prec k\) is
reachable from \(v\). Hence the first non-zero coordinate of \(\mu(v)\), with
respect to \(\prec\), is the \(k\)-th coordinate.

On the other hand, since \(\mu=\mu(v)\in P_i^\prec(M)\), the first non-zero
coordinate of \(\mu(v)\), with respect to \(\prec\), is the \(i\)-th coordinate.
Therefore \(i=k\). Hence \(v\in P(l_i)\), and so
$\mu\in \mu(P(l_i))$.
Thus
$P_i^\prec(M)\subseteq \mu(P(l_i))$,
and the equality follows.
\end{proof}

In the characterization below we shall use the canonical order \(\prec_M\) induced
by the maximum vector \(\mu_\rho\). Thus, for the rest of this section, when
\(M\) has a maximum element and no order is explicitly specified, we write
$P_i(M):=P_i^{\prec_M}(M)$.

We now define the vector condition that will characterize \(\mu\)-representations of
tree-child networks. The definition is designed to mirror the ordered tree-path
decomposition of the previous section: the sets \(P_i(M)\) will play the role of the
paths, and the differences between consecutive vectors will encode the bridges.

\begin{defn}\label{def:tc-mu-compatible}
Let \(M\subseteq\mathbb N^n\) be a finite set of vectors. We say that \(M\) is
\emph{tree-child \(\mu\)-compatible} if the following conditions hold.

\begin{enumerate}
    \item Every vector in \(M\) is non-zero, and
    $\delta_1,\ldots,\delta_n\in M$.

    \item \(M\) has a maximum element for the coordinatewise order. We denote it by
    $\mu_\rho$.

    \item Let \(\prec_M\) be the canonical order induced by \(\mu_\rho\), and let
    $\mathcal P_M=\{P_i(M):i\in[n]\}$
    be the partition defined above. For each \(i \in [n]\), the set \(P_i(M)\) is a chain
    for the coordinatewise order. More precisely, its elements can be written as
    $    P_i(M)=
    \{\mu_i^1,\mu_i^2,\ldots,\mu_i^{m_i}\}$,
    with
    \(\mu_i^1=\delta_i\) and \(\mu_i^1<\mu_i^2<\cdots<\mu_i^{m_i}\).
    We denote the maximal element of this chain by
    $ \rho_i:=\mu_i^{m_i}$.

    \item For every \(i\in[n]\) and every
    \(j=2,\ldots,m_i\), the difference between two consecutive vectors in the
    chain \(P_i(M)\) is a sum of maximal elements of later chains. That is, there
    
    exists a non-empty subset of \(\{k:i\prec_M k\}\), which we fix and denote by
    $B_i^j$ such that
    $\mu_i^j-\mu_i^{j-1}
    =
    \sum_{k\in B_i^j}\rho_k$.

\end{enumerate}
\end{defn}


\begin{lem}\label{lem:condition5} 
Under the assumptions of Definition \ref{def:tc-mu-compatible}, let \(i_\rho\) be the unique index such that \(\mu_\rho\in P_{i_\rho}(M)\). Then every maximal vector \(\rho_k\), with \(k\neq i_\rho\), appears in at least one of the sums in condition \((4)\). Equivalently, for every \(k\neq i_\rho\), there exist \(i\in[n]\) and \(j\in\{2,\ldots,m_i\}\) such that \(k\in B_i^j\).
\end{lem}

\begin{proof}
Suppose, for a contradiction, that \(\rho_k\), with \(k\neq i_\rho\), does not appear in any of the sums in condition \((4)\). Since \(\delta_h\leq\mu_\rho\) for every \(h\in[n]\), all coordinates of \(\mu_\rho\) are strictly positive. Hence \(i_\rho\) is the first index with respect to \(\prec_M\), and in particular \(i_\rho\prec_M k\).

We prove, by reverse induction on the indices preceding \(k\) in \(\prec_M\), that every vector in \(P_i(M)\) has zero \(k\)-th coordinate. Let \(i\prec_M k\), and assume that \(\rho_h^{(k)}=0\) for every \(h\) such that \(i\prec_M h\prec_M k\). For every \(h\in B_i^j\), we have \(i\prec_M h\). The case \(h=k\) is excluded by assumption; if \(i\prec_M h\prec_M k\), then \(\rho_h^{(k)}=0\) by the induction hypothesis; and if \(k\prec_M h\), then \(\rho_h^{(k)}=0\) by the definition of \(P_h(M)\). Therefore, every maximal vector occurring in
\[
\mu_i^j-\mu_i^{j-1}=\sum_{h\in B_i^j}\rho_h
\]
has zero \(k\)-th coordinate. Since \(\delta_i^{(k)}=0\), it follows that every vector in \(P_i(M)\) has zero \(k\)-th coordinate.

Applying this to \(i=i_\rho\), we obtain \(\mu_\rho^{(k)}=\rho_{i_\rho}^{(k)}=0\), contradicting \(\delta_k\leq\mu_\rho\). Hence \(\rho_k\) must appear in at least one of the sums in condition \((4)\).
\end{proof}

\begin{rmk}
The conditions in Definition~\ref{def:tc-mu-compatible} have different roles.
Conditions \((1)\) and \((2)\) identify the candidate leaves and root vector. Condition
\((3)\) is the \(\mu\)-analogue of the path partition. Condition \((4)\) encodes the
bridges between paths. 
\end{rmk}

\begin{rmk}
Notice that, by Lemma~\ref{lem:mu-partition}, the collection $\mathcal P_M=\{P_i(M):i\in[n]\}$ being a partition is not an additional condition in
Definition~\ref{def:tc-mu-compatible},  condition $(3)$.
\end{rmk}

\begin{defn}\label{def:mu-reconstruction-digraph}
Let \(M\subseteq\mathbb N^n\) be tree-child \(\mu\)-compatible. For the fixed choices of the subsets \(B_i^j\) in condition \((4)\), the
\emph{reconstruction digraph} associated with \(M\) is the directed graph $\N(M)=(M,E_M)$ 
whose vertex set is \(M\), and whose arc set is defined as follows.

\begin{enumerate}
    \item For every \(i\in [n]\) and every \(j=2,\ldots,m_i\), we add the
     arc $(\mu_i^j,\mu_i^{j-1})$.
    
    \item For every \(i\in [n]\), every \(j=2,\ldots,m_i\), and every
    \(k\in B_i^j\), we add the arc
    $(\mu_i^j,\rho_k)$.
\end{enumerate}
\end{defn}

\begin{prop}\label{prop:reconstruction-is-tc-network}
Let \(M\subseteq\mathbb N^n\) be tree-child \(\mu\)-compatible. Then the
reconstruction digraph
\(\N(M)=(M,E_M)\),
together with the leaf-labelling
\(\delta_i \mapsto i\) for \(i\in[n]\),
is a tree-child phylogenetic network on \([n]\). 
Moreover, \(\N(M)\) has no elementary paths.
\end{prop}

\begin{proof}
We first prove that \(\N(M)\) is a rooted directed acyclic graph.

By condition \((2)\) in Definition~\ref{def:tc-mu-compatible}, \(M\) has a maximum
element for the coordinatewise order, denoted by \(\mu_\rho\). Since
$\delta_1,\ldots,\delta_n\in M$
by condition \((1)\), and \(\delta_i\leq \mu_\rho\) for every \(i\), all coordinates
of \(\mu_\rho\) are strictly positive. Hence \(\mu_\rho\) belongs to the set \(P_{i_\rho}(M)\),
where \(i_\rho\) is the first index with respect to the canonical order
\(\prec_M\). Since \(\mu_\rho\) is the maximum element of \(M\), it is also the
maximal element of the chain \(P_{i_\rho}(M)\), that is,
$\rho_{i_\rho}=\mu_\rho$.

By construction, the arcs of \(\N(M)\) are of two types. The arcs of the form
$(\mu_i^j,\mu_i^{j-1})$ 
 go downwards inside a chain. The rest of the arcs have the form
$(\mu_i^j,\rho_k)$, with 
$i\prec_M k$,
and therefore go from one chain to a later chain. Hence \(\N(M)\) is acyclic.

We now identify the root. The vector \(\mu_\rho=\rho_{i_\rho}\) has in-degree zero.
Indeed, it is the maximal element of the first chain, and bridge arcs only enter maximal
vectors of later chains. Thus no arc can enter \(\mu_\rho\). Moreover, we now prove that this root is unique. Let \(\rho_k\) be the maximal element of a chain \(P_k(M)\) with
\(k\neq i_\rho\). By Lemma \ref{lem:condition5},  there exist \(i\) and \(j\) such that
$k\in B_i^j$.
Therefore the bridge arc
$(\mu_i^j,\rho_k)$
belongs to \(E_M\), and so \(\rho_k\) has positive in-degree. Every non-maximal
element of a chain has positive in-degree because it is entered by the corresponding
internal arc of that chain. Therefore \(\mu_\rho\) is the unique vertex of
\(\N(M)\) with in-degree zero.

We also have that every vertex is reachable from \(\mu_\rho\). Indeed, since \(\N(M)\) is finite and acyclic, starting from any vertex and repeatedly following an incoming arc must eventually reach a vertex of in-degree zero. Since \(\mu_\rho\) is the unique vertex of \(\N(M)\) with in-degree zero, every vertex is reachable from \(\mu_\rho\). Thus \(\N(M)\) is a rooted DAG with root \(\mu_\rho\).

We now identify the leaves. For each \(i\), the vector
$\delta_i=\mu_i^1$
has no outgoing internal arc, and no bridge is defined from \(\mu_i^1\). Hence
\(\delta_i\) is a leaf. Conversely, if \(\mu_i^j\) is not equal to \(\delta_i\), then
\(j\geq 2\), and the internal arc
$(\mu_i^j,\mu_i^{j-1})$
belongs to \(E_M\). Hence \(\mu_i^j\) is not a leaf. Therefore the leaves of
\(\N(M)\) are precisely
$\delta_1,\ldots,\delta_n$, which are labelled by \(1,2,\ldots,n\), respectively.

It remains to prove that \(\N(M)\) is tree-child. Let \(\mu_i^j\) be an inner
vertex. Then \(j\geq 2\), and by construction
$(\mu_i^j,\mu_i^{j-1})\in E_M$.
We claim that \(\mu_i^{j-1}\) is a tree vertex. Since \(\mu_i^{j-1}\) is not the
maximal element of its chain, no bridge can enter it. Moreover, the only internal
arc entering it is
$(\mu_i^j,\mu_i^{j-1})$.
Therefore
$\indeg(\mu_i^{j-1})=1$, 
so \(\mu_i^{j-1}\) is a tree vertex. Hence every inner vertex has a tree child, and
\(\N(M)\) is tree-child.

Finally, we show that \(\N(M)\) has no elementary paths. Let
$(\mu_i^j,\mu_i^{j-1})$
be an internal arc. Since \(j\geq 2\), condition \((4)\) gives
$\mu_i^j-\mu_i^{j-1}
=
\sum_{k\in B_i^j}\rho_k$.
The left-hand side is non-zero because
$\mu_i^{j-1}<\mu_i^j$.
Hence \(B_i^j\neq\emptyset\). Therefore \(\mu_i^j\) has at least one bridge as an
additional outgoing arc, besides the internal arc to \(\mu_i^{j-1}\). Thus
$\outdeg(\mu_i^j)\geq 2$.
Every arc of \(\N(M)\) has a vertex of this form as its tail. Consequently, no arc has a tail of out-degree one, and thus \(\N(M)\) has no elementary paths.
\end{proof}

\begin{prop}\label{prop:reconstruction-has-mu-set}
Let \(M\subseteq\mathbb N^n\) be tree-child \(\mu\)-compatible. Then
$\mu(\N(M))=M$, where \(\N(M)\) is endowed with the leaf-labelling \(\delta_i \mapsto i\).
\end{prop}

\begin{proof}
Let $\N(M)=(M,E_M)$ be the reconstruction digraph. For a vertex \(\mu\in M\), let
$\mu_{\N(M)}(\mu)$
denote its \(\mu\)-vector in the network \(\N(M)\). We shall prove that
$\mu_{\N(M)}(\mu)=\mu$
for every \(\mu\in M\). This will imply \(\mu(\N(M))=M\).

Recall that, by Definition~\ref{def:tc-mu-compatible}, the canonical partition of
\(M\) is
$
\mathcal P_M=\{P_i(M):i \in [n]\}$,
where
$P_i(M)=\{\mu_i^1,\mu_i^2,\ldots,\mu_i^{m_i}\}$,
 with 
$\delta_i=\mu_i^1<\mu_i^2<\cdots<\mu_i^{m_i}=\rho_i$.

We prove the claim by reverse induction on the order \(\prec_M\) of the chains, and
inside each chain by induction on \(j\).
Let \(i\) be fixed, and suppose that the claim has already been proved for all
vertices belonging to chains \(P_k(M)\) with
$i\prec_M k$.
We prove it for the vertices of \(P_i(M)\).

For \(j=1\), we have
$\mu_i^1=\delta_i$.
By Proposition~\ref{prop:reconstruction-is-tc-network}, the leaves of \(\N(M)\) are
precisely \(\delta_1,\ldots,\delta_n\), which are associated respectively to labels \(1, \ldots, n\). Hence
$\mu_{\N(M)}(\mu_i^1)=\mu_{\N(M)}(\delta_i)=\delta_i=\mu_i^1$.

Now let \(j\geq 2\), and assume that
$\mu_{\N(M)}(\mu_i^{j-1})=\mu_i^{j-1}$.
By the definition of $\N(M)$, the children of \(\mu_i^j\) are
$\mu_i^{j-1}$
together with the vertices
$\rho_k$ for $k\in B_i^j$.
Therefore, 
by Lemma \ref{lem:children_mu},
$
\mu_{\N(M)}(\mu_i^j)
=
\mu_{\N(M)}(\mu_i^{j-1})
+
\sum_{k\in B_i^j}\mu_{\N(M)}(\rho_k)$.

By the induction hypothesis inside the chain,
$\mu_{\N(M)}(\mu_i^{j-1})=\mu_i^{j-1}$.
Moreover, for every \(k\in B_i^j\), we have \(i\prec_M k\), and hence, by the
reverse induction hypothesis on the chains,
$\mu_{\N(M)}(\rho_k)=\rho_k$.
Thus
$
\mu_{\N(M)}(\mu_i^j)
=
\mu_i^{j-1}
+
\sum_{k\in B_i^j}\rho_k$.
By condition \((4)\) in Definition~\ref{def:tc-mu-compatible},
$
\mu_i^j-\mu_i^{j-1}
=
\sum_{k\in B_i^j}\rho_k$.
Hence
$\mu_{\N(M)}(\mu_i^j)=\mu_i^j$.

This proves the claim for all vertices in \(P_i(M)\). By reverse induction over the
chains, the claim holds for every vertex of \(M\). 
\end{proof}

We have proved that every tree-child \(\mu\)-compatible set of vectors gives rise to
a tree-child phylogenetic network whose \(\mu\)-representation is the original set.
We now prove the converse: every \(\mu\)-representation of a tree-child network
satisfies the compatibility conditions of Definition \ref{def:tc-mu-compatible}.

\begin{prop}\label{prop:mu-set-is-compatible}
Let \(\N\) be a tree-child phylogenetic network on \([n]\) without elementary paths, and let
$M=\mu(\N)$.
Then \(M\) is tree-child \(\mu\)-compatible.
\end{prop}

\begin{proof}
We verify the conditions in Definition~\ref{def:tc-mu-compatible}. First, every vector in \(M\) is non-zero, because every vertex of a phylogenetic
network is reachable from the root and, since the graph is finite and acyclic, every
vertex has at least one descendant leaf. Moreover, for each leaf \(l_i\), we have $\mu(l_i)=\delta_i$ since we use the canonical labelling of the leaves,
and hence $\delta_1,\ldots,\delta_n\in M$.

Let \(\rho\) be the root of \(\N\). We claim that
$\mu_\rho:=\mu(\rho)$
is the maximum element of \(M\) for the coordinatewise order. Indeed, for every
vertex \(v\in V(\N)\), every directed path from \(v\) to a leaf can be extended
backwards to a directed path from \(\rho\) to that same leaf, because \(v\) is
reachable from \(\rho\). Therefore
$\mu(v)\leq \mu(\rho)$.


We now prove condition \((3)\).
Let \(\prec_M\) be the canonical order induced by \(\mu_\rho\). By
Lemma~\ref{lem:mu-partition}, the sets
\[
P_i(M)
=
\{
\mu\in M:
\mu^{(i)}>0
\text{ and }
\mu^{(h)}=0
\text{ for every }h\prec_M i
\}
\]
form a partition of \(M\).

We now use the ordered tree-path structure of \(\N\). By Lemma \ref{lem:mu-is-tree-child-order}, \(\prec_M\) is a tree-path order on \([n]\) for \(\N\). By Proposition \ref{prop:ordering-is-decomposition}, \((\prec_M, \mathcal{P}_{\N, \prec_M})\) is an ordered tree-path decomposition of \(\N\), where
\[
\mathcal{P}_{\N, \prec_M}=\{P_{\prec_M}(l_i):i\in[n]\}.
\]
By Proposition~\ref{prop:mu-partition-from-ordered-decomposition},
$
\mu(P_{\prec_M}(l_i))=P_i(M)$
 for every $i\in[n]$.
Since \(\N\) has no elementary paths, by Proposition \ref{prop:mu_is_set}, distinct vertices have distinct \(\mu\)-vectors. Hence the image under \(\mu\) of each path is a chain for the coordinatewise order. Therefore, for every \(i\), we can write
\[
P_i(M)=\{\mu_i^1,\mu_i^2,\ldots,\mu_i^{m_i}\},
\qquad
\delta_i=\mu_i^1<\mu_i^2<\cdots<\mu_i^{m_i},
\]
and denote \(\rho_i:=\mu_i^{m_i}\). Moreover, we can write
$P_{\prec_M}(l_i)=(v_i^{m_i},\ldots,v_i^1=l_i)
$
so that \(\mu(v_i^j)=\mu_i^j\) for every \(j\).

We now prove condition \((4)\). Let \(i\in[n]\) and \(j=2,\ldots,m_i\). In the ordered tree-path decomposition, \(v_i^{j-1}\) is the tree child of \(v_i^j\) lying in the same path. All other children of \(v_i^j\), if any, are the maximal vertices of later paths. Let \(B_i^j\) be the set of indices \(k\) such that there is a bridge arc \((v_i^j,v_k^{m_k})\). By the definition of an ordered tree-path decomposition,
$B_i^j\subseteq\{k:i\prec_M k\}$.
Moreover, \(B_i^j\neq\emptyset\). Indeed, if \(B_i^j=\emptyset\), then \(v_i^j\) would have the unique child \(v_i^{j-1}\), which is a tree vertex, and the arc \((v_i^j,v_i^{j-1})\) would be elementary, contrary to the hypothesis.
Moreover, \(\mu(v_k^{m_k})=\rho_k\). Hence, by Lemma \ref{lem:children_mu},
\[
\mu_i^j
=
\mu_i^{j-1}
+
\sum_{k\in B_i^j}\rho_k.
\]
Equivalently,
\[
\mu_i^j-\mu_i^{j-1}
=
\sum_{k\in B_i^j}\rho_k.
\]


Therefore \(M\) is tree-child \(\mu\)-compatible.
\end{proof}

\begin{thm}
\label{thm:tc-mu-characterization}
Let \(M\subseteq\mathbb N^n\) be a finite set of vectors. Then \(M\) is the
\(\mu\)-representation of a tree-child phylogenetic network on \([n]\) without elementary paths 
if and only if \(M\) is tree-child \(\mu\)-compatible.

Moreover, if such a network exists, then it is unique up to isomorphism.
\end{thm}

\begin{proof}
Assume first that \(M\) is tree-child \(\mu\)-compatible. By
Proposition~\ref{prop:reconstruction-is-tc-network}, the reconstruction digraph
$\N(M)$, together with the leaf labelling \(\delta_i \mapsto i\),
is a tree-child phylogenetic network without elementary paths on $[n]$. Moreover, by
Proposition~\ref{prop:reconstruction-has-mu-set}, we have $\mu(\N(M))=M$.
Thus \(M\) is the \(\mu\)-representation of a tree-child phylogenetic network without
elementary paths.

Conversely, suppose that there exists a tree-child phylogenetic network \(\N\) on \([n]\)
without elementary paths such that
$M=\mu(\N)$.
Then, by Proposition~\ref{prop:mu-set-is-compatible}, the set \(M\) is tree-child
\(\mu\)-compatible.

Finally, suppose that \(\N\) and \(\N'\) are two tree-child phylogenetic networks on \([n]\)
without elementary paths such that
$\mu(\N)=M=\mu(\N')$.
By Theorem~\ref{thm:mu:cardona-isomorphism}, $
\N\cong \N'$.
Therefore, the realizing network is unique up to isomorphism.
\end{proof}




\section{Conclusion}\label{sec:Conc}

We have characterized tree-child phylogenetic networks by ordered tree-path decompositions. For networks without elementary paths, translating this structure into \(\mathbb N^n\) gives the notion of tree-child \(\mu\)-compatibility. Our main theorem proves that tree-child \(\mu\)-compatibility is equivalent to being realizable as the \(\mu\)-representation of a tree-child phylogenetic network. Together with the known completeness of the \(\mu\)-representation for tree-child networks, this gives a complete description of the feasible \(\mu\)-space for this class.

This characterization provides the basis for several future developments,
some of which we will present in subsequent versions of this work. In particular, 
we will use the chracterization to develop 
procedures to randomly generate tree-child networks. 
We will also investigate ways to compute a consensus for a
collection of tree-child phylogenetic networks motivated in part by results in
\cite{huber2021phylogenetic}.

\vfill

\noindent \sloppy \textbf{Acknowledgements.} 
J.C.P., T.F., and V.M. was supported by Grant PID2025-169396NA-I00 funded by MICIU/AEI/10.13039/501100011033 and by ERDF/EU.
J.C.P., V.M. and K.T.H. would also like to thank
the organisers of the events
`Mathematics of Evolution-Phylogenetic Trees and Networks',
Institute of Mathematical Sciences, Singapore, 2023, and 
`Uniting Phylogenetic Network Research', Lorentz Center, Netherlands, 2026, 
in which they started to work on and 
then develop the ideas in this paper.

\bibliographystyle{alpha}
\bibliography{mybibfile}

\end{document}